\documentclass[11pt,a4paper]{article}
\usepackage{amsfonts}
\usepackage{color}
\usepackage{latexsym}
\usepackage{amsmath}
\usepackage{amssymb}
\usepackage{amsthm}
\usepackage{graphicx}
\usepackage[toc,page]{appendix}
\usepackage{mathrsfs} 
\usepackage{textcomp}

\usepackage[left=3.2cm,top=2.5cm,right=3.2cm,bottom=3cm]{geometry}

\newcommand{\IR}{{\mathbb R}}

\newcommand{\IS}{{\mathbb S}}

\newcommand{\inner}[2]{ \ensuremath { \langle {#1},{#2} \rangle } }
\newcommand{\vol}{ \ensuremath { {\rm Vol} } }
\newcommand{\os}{ \ensuremath { +_{\!\varphi} } }
\newcommand{\sub}[1]{ \ensuremath { _{_{\! #1 }}\! } }

\numberwithin{equation}{section}

\newtheoremstyle{Nikos1theoremstyle} 
    {\topsep}    
    {}      
    {}           
    {+2pt}       
    {\scshape}   
    {.}        
    {.5em}       
    {}           

\newtheorem{theorem}{Theorem}[section]
\newtheorem{lemma}[theorem]{Lemma}
\newtheorem{proposition}[theorem]{Proposition}

\newtheorem{definition}[theorem]{Definition}

\begin{document}

\title{Orlicz Mixed Affine Quermassintegrals}

\author{Nikos Dafnis
}

\date{}

\maketitle

\begin{abstract}
  Lutwak's notion of affine quermassintegrals of a convex body quickly
 became of great importance in convex and affine geometry and more 
 recently, also in asymptotic geometric analysis. In this note, 
 following the ideas in \cite{Zhao}, we introduce the notion of the 
 Orlicz mixed affine quermassintegrals of a convex body in $\IR^n$, 
 as a generalization of the affine quermassintegrals in the framework 
 of the Orlicz-Brunn-Minkowski theory. We prove a Minkowski inequality 
 for the Orlicz mixed and affine quermassintegrals, and an 
 Orlicz-Brunn-Minkowski inequality, which provides of a direct 
 generalization of Lutwak's Brunn-Minkowski inequality for affine 
 quermassintegrals, in the Orlicz space.
\end{abstract}


\section{Introduction.}\label{sec.intro}

 Let ${\cal K}^n$ be the class of all non-empty compact convex subsets of 
$\IR^n$.  The support function of a $K\in{\cal K}^n$ is defined by 
\begin{align*}
 h_K(x)=\sup_{y\in K} \inner{x}{y}, \quad x\in\IR^n.
\end{align*}
Support functions are sublinear i.e., subadditive and homogeneous of degree 
$1$, and therefor often regarded as functions on $\IS^{n-1}$. Conversely, 
any sublinear real-valued function on $\IR^n$, is the support function of 
a unique compact convex set. Consequently, any $K\in{\cal K}^n$ is uniquely 
determined by its support function.

The classical Brunn-Minkowski theory combine two basic concepts in 
geometry, volume and Minkowski (vector) addition. More 
precisely, if $K,L\in{\cal K}^n$, $a,b>0$, their Minkowski 
linear combination $aK+bL\in{\cal K}^n$ can be defined by
\begin{align*}
 h_{aK+bL}(u)=ah_K(u)+bh_L(u), \quad u\in\IS^{n-1}.
\end{align*}
The Brunn-Minkowski theory played a crucial role in the development 
of convex geometry and asymptotic geometric analysis, among other 
areas of mathematics. Fundamental within the theory, is the Brunn-Minkowski 
inequality:
\begin{align*}
 V(K+L)^\frac1n\geq V(K)^\frac1n+V(L)^\frac1n,
\end{align*}
for any $K,L\in{\cal K}^n$, and if $K$ and $L$ are non-trivial, then
equality holds if and only if $K$ and $L$ are homothetic, or they lie in 
parallel hyperplanes. 




Minkowski's fundamental notion of mixed volume $V_1(K,L)$, 
$K,L\in{\cal K}^n$, is defined to be proportional to the first variation 
of the volume with respect to Minkowski linear combination, by the formula
\begin{align*}
 V_1(K,L)=\frac1n\lim_{\varepsilon\to0^+}
 \frac{V(K+\varepsilon L)-V(K)}{\varepsilon}.
\end{align*}

 Aleksandrov \cite{Aleks} and Fenchel and Jensen \cite{FJ} proved
that or any $K\in{\cal K}^n$ there exists a unique Borel measure 
$S(K,\cdot)$  on the unit sphere $\IS^{n-1}$, called the surface area 
measure of $K$, such that for any $L\in{\cal K}^n$ their mixed volume 
has the following integral representation:
\begin{align*}
 V_1(K,L)=\frac1n\int_{\IS^{n-1}} h(L,u)\,dS(K,u).
\end{align*}

By its definition the mixed volume satisfy that  $V_1(K,K)=V(K)$ for any 
$K\in{\cal K}^n$. The Minkowski inequality settle in general the relation 
between the volume and the mixed volume: 
\begin{align*}
 V_1(K,L)^n\geq V(K)^{n-1}V(L),
\end{align*}
$K,L\in{\cal K}^n$. If moreover $K$ and $L$ are non-trivial, then 
equality holds if and only if $K$ and $L$ are homothetic, or they 
lie in parallel hyperplanes. 

\medskip

\noindent {\sc $L_p$ Mixed Volume}.

\smallskip

Let ${\cal K}_o^n$ be the class of all nonempty compact convex subsets of 
$\IR^n$, that contain the origin in their interior.  
Firey \cite{Firey1}, \cite{Firey2} introduced a new notion of linear 
combination of convex bodies in ${\cal K}_o^n$: For $p\geq 1$, 
$K,L\in{\cal K}_o^n$ and $a,b>0$, their $L_p$ linear combination is the set 
$a\cdot_{_{^p}}\! K+_pb\cdot_{_{^p}}\!L\in{\cal K}_o^n$, with support function
\begin{align*}
 h_{a\cdot_{_{^p}} K+_p\,b\,\cdot_{_{^p}} L}(u)^p=ah_K(u)^p+bh_L(u)^p,
 \quad u\in\IS^{n-1}.
\end{align*}
See also \cite{LYZ1} for an extension on non-convex sets. 
Using Firey's linear combination, Lutwak \cite{Lutwak7}, \cite{Lutwak5} 
gave rise to the $L_p$-Brunn-Minkowski theory as an extension of the 
Brunn-Minkowski theory, which strengthened many of the classical results. 
In his setting, the $L_p$ mixed volume of $K,L\in{\cal K}_o^n$ can be 
defined by
\begin{align*}
 V_p(K,L)=\frac{p}{n}
 \lim_{\varepsilon\to0^+}\frac{V(K+_p\varepsilon\cdot_{_{^p}} L)-V(K)}
 {\varepsilon},
\end{align*}
 and has the following integral representation
\begin{align*}
 V_p(K,L)=\frac1n\int_{\IS^{n-1}} h_L(u)^p\,h_K(u)^{1-p}\,dS(K,u).
\end{align*}
Note that $V_p(K,K)=V(K)$, for any $K\in{\cal K}_o^n$, and in correspondence 
to the classical theory, the $L_p$-Minkowski inequality asserts that
\begin{align*}
 V_p(K,L)^n\geq V(K)^{n-p}\,V(L)^p,
\end{align*}
for any $K,L\in{\cal K}_{o}^n$, with equality if and only if $K$ and $L$ 
are dilates of each other, 
and as a consequence we get the following $L_p$-Brunn-Minkowski inequality:
\begin{align*}
 V(K+_pL)^\frac{p}n\geq V(K)^\frac{p}n+V(L)^\frac{p}n,
\end{align*}
for any $K,L\in{\cal K}_{o}^n$, with equality if and only if $K$ and $L$ 
are homothetic. 

For further details on Brunn-Minkowski and $L_p$-Brunn-Minkowski 
theories, we refer to the book of Schneider \cite{Schneider}.

\medskip

\noindent  {\sc Orlicz Mixed Volume}.

\smallskip

Lutwak, Yang and Zhang \cite{LYZ2}, \cite{LYZ6} initiated a further extension 
of Brunn-Minkowski theory to an Orlicz setting. This involves the replacement 
of the function $t^p$, by an increasing convex function in 
$\varphi:[0,\infty)\to[0,\infty)$. The new Orlicz-Brunn-Minkowski theory 
studied systematically by Gardner, Hug \& Weil in \cite{GHW1}, where the 
authors constructed a solid framework and indicate its relation to the Orlicz 
spaces. Following their point of view, we denote by $\mathscr{C}$, the class 
of all increasing convex functions $\varphi:[0,\infty)\to[0,\infty)$ with 
$\varphi(0)=0$ and $\varphi(1)=1$. 

For every $K,L\in{\cal K}_o^n$, 
$a,b>0$ and $\varphi\in\mathscr{C}$, the \textit{Orlicz linear combination} 
$a\cdot\sub{^\varphi} K +_\varphi b\cdot\sub{^\varphi} L\in{\cal K}_o^n$ or 
simply $a\cdot K +_\varphi b\cdot L$ is defined by
\begin{align}\label{eq.Orlicz_sum}
 h_{a\cdot K +_\varphi b\cdot L}(x)
 =\inf\left\{\lambda>0:\,a\,\varphi\left(\frac{h_K(x)}{\lambda}\right) 
  + b\,\varphi\left(\frac{h_L(x)}{\lambda}\right) \leq 1\right\},
\end{align}
or equivalently by the implicit equation
\begin{align}\label{eq.Orlicz_sum_*}
 a\,\varphi\left(\frac{h_K(x)}{h_{a\cdot K +_\varphi b\cdot L}(x)}\right) + 
 b\,\varphi\left(\frac{h_L(x)}{h_{a\cdot K +_\varphi b\cdot L}(x)}\right)=1.
\end{align}
The Orlicz linear combination is continuous with respect to the Hausdorff
metric. 
In particular, for any $K,L\in{\cal K}_o^n$, we have that 
\begin{align}\label{eq.O_lim}
 K+_\varphi \varepsilon\cdot L \to K,
\end{align}
as $\varepsilon\to0^+$, in the Hausdorff metric 
\begin{align*}
 \delta(K,L)=\sup_{u\in\IS^{n-1}}\big|h_K(u)-h_L(u)\big|.
\end{align*}

The \textit{Orlicz mixed volume} of $K,L\in{\cal K}_{o}^n$, 
$\varphi\in\mathscr{C}$, can be defined by 
\begin{align}\label{eq.Orlicz-mv_1}
 V_{\varphi}(K,L)=\frac1n\int_{\IS^{n-1}}
 \varphi\left(\frac{h_L(u)}{h_K(u)}\right)h_K(u)\,dS(K,u),
\end{align}
 As in the $L_p$ case, the Orlicz mixed volume of $K,L\in{\cal K}_{o}^n$, 
is proportional to the first variation of volume with respect to their Orlicz 
linear combination:
\begin{align}\label{eq.Orlicz-mv_1_*}
 V_\varphi(K,L)=\frac{\varphi'(1^-)}{n}
 \lim_{\varepsilon\to0^+}\frac{V(K+_\varphi\varepsilon\cdot L)
 -V(K)}{\varepsilon},
\end{align}
where $\varphi'(1^-)$ is the left derivative of $\varphi$ at $1$.
Similarly to the $L_p$ case, we have that $V_\varphi(K,K)=V(K)$ for any 
$K\in{\cal K}_o^n$, and more generally the Orlicz-Minkowski inequality
asserts that
\begin{align}\label{eq.Orlicz-Mink_ineq}
 \frac{V_\varphi(K,L)}{V(K)}\geq
 \varphi\left(\left(\frac{V(L)}{V(K)}\right)^{1/n}\right),
\end{align}
for all $K,L\in{\cal K}_{o}^n$. If $\varphi$ is strictly convex then 
equality holds if and only if $K$ and $L$ are dilates of each other. 
We also have a Orlicz-Brunn-Minkowski inequality, which relates the volume 
with the Orlicz linear combination:
\begin{align}\label{eq.Orlicz-B-M_ineq}
 1\geq \varphi\left(\left(\frac{V(K)}{V(K\os L)}\right)^{1/n}\right)
 +\varphi\left(\left(\frac{V(L)}{V(K\os L)}\right)^{1/n}\right),
\end{align}
for all $K,L\in{\cal K}_o^n$, and if $\varphi$ is strictly convex then 
equality holds if and only if $K$ and $L$ are dilates of each other. 
We refer to \cite{GHW1} for any further details on the Orlicz-Brunn-Minkowski 
theory.

\medskip

\noindent  {\sc Affine quermassintegrals}.

\smallskip

Lutwak \cite{Lutwak1} defined the affine quermassintegrals of a convex body 
$K$ in $\IR^n$ by
\begin{align*}
 \Phi_{n-j}(K) = \frac{\omega_n}{\omega_j}\left(\int_{G_{n,j}}
 \vol_j\big(K|\xi\big)^{-n}\,d\nu_{n,j}(\xi)\right)^{-\frac{1}{n}},
\end{align*}
for $1<j<n$, while $\Phi_0(K)=V(K)$ and $\Phi_n(K)=\omega_n$. Here and 
for the rest of this note, $\nu_{n,j}$ is the Haar probability measure on 
the Grassmannian manifold $G_{n,j}$ of the $j$-dimensional subspaces of 
$\IR^n$. The terminology ``affine'' was justified a few years later 
by Grinberg \cite{Grinberg}, how showed that these quantities are actually 
invariant under volume preserving linear transformations. Lutwak proved a 
Brunn-Minkowski inequality for the affine quermassintegrals:
\begin{align}\label{eq.B-M-AQ}
 \Phi_{n-j}(K+L)^{1/j}\geq\Phi_{n-j}(K)^{1/j}+\Phi_{n-j}(L)^{1/j},
\end{align}
and conjectured in  \cite{Lutwak8} that they satisfy the inequalities
\begin{align*}
 \omega_n^{n-k}\Phi_{n-j}(K)^{k} \geq \omega_n^{n-j}\Phi_{n-k}(K)^{j},
\end{align*}
for all $0<j<k\leq n$. In particular, Lutwak asks if the following 
inequalities holds true:
\begin{align}\label{eq.Lutwak-conjecture}
 \Phi_{n-j}(K)^n\geq \omega_n^{n-j}\,V(K)^{j},
\end{align}
for every $0\leq j <n$, with equality if and only if $K$ is an ellipsoid.
Most of these conjectures remain open. Note that two cases of 
\eqref{eq.Lutwak-conjecture} follow  from classical results: when $j=n-1$ 
this inequality is the Petty projection inequality and when $j=1$ and $K$ is 
origin symmetric then \eqref{eq.Lutwak-conjecture} is the Blaschke-Santal\'{o} 
inequality. For more details we refer to the book of Gardner \cite{Gardner} 
(see also \cite{DP} and \cite{PP}, where an asymptotic version of 
\eqref{eq.Lutwak-conjecture} is proved).

\medskip

In \cite{LZX}, an extension of affine quermassintegrals was considered, where
the authors defined the Orlicz mixed affine quermassintegrals for 
$K,L\in{\cal K}_o^n$, $\varphi\in\mathscr{C}$ and $0 < j \leq n$, by
\begin{align*}
 \varPhi_{\varphi,n-j}(K,L):=\frac{\omega_n}{\omega_j}\left(\int_{G_{n,j}}
 V_\varphi^{(j)}\big(K|\xi,L|\xi\big)^{-n}\,d\nu_{n,j}(\xi)\right)^{-\frac1n}.
\end{align*}
These quantities are invariant under volume preserving linear transformations,
and provide a generalization of affine quermassintegrals to the Orlicz spaces.

\medskip

In this note, drawing our inspiration from \cite{Zhao}, we introduce the 
following alternative definition. For any $\varphi\in\mathscr{C}$ and $0 < j \leq n$,
we define the \textit{Orlicz mixed affine quermassintegrals} of 
$K,L\in{\cal K}_o^n$, by
\begin{align*}
 \Phi_{\varphi,n-j}(K,L):=\frac{\omega_n}{\omega_j}\left(\int_{G_{n,j}}
 V_\varphi^{(j)}\big(K|\xi,L|\xi\big)\,\vol_j
 \big(K|\xi\big)^{-n-1}\,d\nu_{n,j}(\xi)\right)^{-\frac1n}.
\end{align*}
Our definition of Orlicz mixed affine quermassintegrals combine and extend 
both Lutwak's concepts of affine quermassintegrals and Orlicz mixed volume. 

In section \ref{sec.OMAQ} we prove their invariance under volume 
preserving linear transformations and we show a first variation 
formula for them (see Proposition \ref{prop.1st_variation}) with respect 
to the Orlicz linear combination:
\begin{align}\label{eq.1st_variation_intro}
  \Phi_{\varphi,n-j}(K,L)^{-n}=
  \frac{\varphi'(1^-)}{j\,\Phi_{n-j}(K)^{n+1}}\,\lim_{\varepsilon\to0^+}
  \frac{\Phi_{n-j}(K+_\varphi\varepsilon\cdot\sub{^\varphi} L)
  -\Phi_{n-j}(K)}{\varepsilon}.
 \end{align}
 This can be seen as a generalization of the corresponding formula 
\eqref{eq.Orlicz-mv_1} for the Orlicz mixed volume.

 In section \ref{sec.O-M_OMAQ} we use H\"{o}lder's inequality to derive 
the following Orlicz-Minkowski inequality for the Orlicz mixed affine 
quermassintegrals
\begin{align}\label{eq.O-M_AQ_intro}
  \left(\frac{\Phi_{\varphi,n-j}(K,L)}{\Phi_{n-j}(K)}\right)^{-n}\geq 
  \varphi\left(\left(\frac{\Phi_{n-j}(L)}{\Phi_{n-j}(K)}\right)^{1/j}\right),
 \end{align}
which generalize the Orlicz-Minkowski inequality \eqref{eq.Orlicz-Mink_ineq}
for the mixed volume.
 
In the last section \ref{sec.O-B-M_AQ}, we prove an Orlicz-Brunn-Minkowski 
inequality for Lutwak's affine quermassintegrals:
\begin{align}\label{eq.O-B-M_AQ_intro}
  1\geq 
  \varphi\left(\left(\frac{\Phi_{n-j}(K)}
  {\Phi_{n-j}(K\os\varepsilon\cdot L)}\right)^{1/j}\right)
  +\varepsilon
  \varphi\left(\left(\frac{\Phi_{n-j}(L)}
  {\Phi_{n-j}(K\os\varepsilon\cdot L)}\right)^{1/j}\right).
 \end{align}
Note that \eqref{eq.O-B-M_AQ_intro} generalize the corresponding 
Brunn-Minkowski inequalities for Orlicz mixed volumes 
\eqref{eq.Orlicz-Mink_ineq}, and for affine quermassintegrals 
\eqref{eq.B-M-AQ}.

\medskip

\noindent {\bf Acknowledgments}. 
The author is supported by the Austrian Science Fund (FWF): 
Lise Meitner [M 2338-N35].

\section{Orlicz Mixed Affine Quermassintegrals}\label{sec.OMAQ}

\begin{definition}[Orlicz mixed Affine Quermassintegrals]
\label{def.O-Aff-Querm}
 The \textit{Orlicz mixed affine quermassintegrals} are defined for any 
 $K,L\in{\cal K}_{o}^n$, $\varphi\in\mathscr{C}$ and $1\leq j \leq n$, by
 \begin{align}\label{eq.O-Aff-Querm}
 \Phi_{\varphi,n-j}(K,L):=\frac{\omega_n}{\omega_j}\left(\int_{G_{n,j}}
 V_\varphi^{(j)}\big(K|\xi,L|\xi\big)\,\vol_j
 \big(K|\xi\big)^{-n-1}\,d\nu_{n,j}(\xi)\right)^{-\frac1n}.
 \end{align}
\end{definition}

\begin{lemma}
 For any $K\in{\cal K}_{o}^n$, $\lambda>0$ and all $1\leq j \leq n$,
 we have that
 \begin{align}\label{eq.K,K}
  \Phi_{\varphi,n-j}(K,\lambda K)=\varphi(\lambda)^{-1/n}\,\Phi_{n-j}(K).
 \end{align}
\end{lemma}
\begin{proof}
 By the definition of Orlicz mixed volume \eqref{eq.Orlicz-mv_1}, 
 we have that
 \begin{align*}
  V_\varphi(K,\lambda K)
  &=\frac1n\int_{\IS^{n-1}}
   \varphi\left(\frac{h_{\lambda K}(u)}{h_K(u)}\right)h_K(u)\,dS(K,u) \\
  &=\varphi(\lambda)\,\frac1n\int_{\IS^{n-1}} h_K(u)\,dS(K,u) \\
  &=\varphi(\lambda)\,V_1(K,K)=\varphi(\lambda)\,V(K).
 \end{align*}
 Thus,
 \begin{align*}
  \Phi_{\varphi,n-j}(K,\lambda K)
  &=\frac{\omega_n}{\omega_j}\left(\int_{G_{n,j}}
    V_\varphi^{(j)}\big(K|\xi,\lambda K|\xi\big)
    \,\vol_j\big(K|\xi\big)^{-n-1}
    \,d\nu_{n,j}(\xi)\right)^{-\frac1n} \\
  &=\varphi(\lambda)^{-\frac1n}\,\Phi_{n-j}(K).
 \end{align*}
\end{proof}

\medskip

 Note that $\varphi(1)=1$, and so we have that
$\Phi_{\varphi,n-j}(K,K)=\Phi_{n-j}(K)$, for all $K\in{\cal K}_o^n$, and 
so in that sense, Orlicz mixed affine quermassintegrals provides a natural 
extension of Lutwak's affine quermassintegrals in the Orlicz setting.

\begin{lemma}
 Let $\varphi\in\mathscr{C}$, $K,L\in{\cal K}_{o}^n$, $\varepsilon>0$ 
 and $\xi\in G_{n,j}$. Then,
 \begin{align*}
  (K+_\varphi\varepsilon\cdot L)\big|\xi 
  = K|\xi +_\varphi \varepsilon\cdot L|\xi.
 \end{align*}
\end{lemma}
\begin{proof}
 For every $u\in\IS^{n-1}\cap\xi$, we have that
 \begin{align}
  h_Q(u)=h_{Q|\xi}(u),
 \end{align}
 for every $Q\in{\cal K}_o^n$. Thus, by \eqref{eq.Orlicz_sum_*} we have
 that for every $u\in\IS\cap\xi$
 \begin{align*}
  &\varphi\left(\frac{h_{K|\xi}(u)}
   {h_{(K+_\varphi\varepsilon\cdot L)|\xi}(u)}\right)
   +\varepsilon\,\varphi\left(\frac{h_{L|\xi}(u)}
   {h_{(K+_\varphi\varepsilon\cdot L)|\xi}(u)}\right) \\
  &=\varphi\left(\frac{h_{K}(u)}
   {h_{K+_\varphi\varepsilon\cdot L}(u)}\right)
   +\varepsilon\,\varphi\left(\frac{h_{L}(u)}
   {h_{K+_\varphi\varepsilon\cdot L}(u)}\right)=1,
 \end{align*}
 which means that
 $(K+_\varphi\varepsilon\cdot L)|\xi=K|\xi+_\varphi\varepsilon\cdot L|\xi$.
\end{proof}

Next we prove a first variation formula for the Orlicz mixed 
affine quermassintegrals with respect to the Orlicz addition.

\begin{proposition}\label{prop.1st_variation}
 Let $\varphi\in\mathscr{C}$, $K,L\in{\cal K}_{o}^n$ 
 and $1\leq j \leq n$. Then,
 \begin{align}\label{eq.1st_variation}
  \Phi_{n-j}(K)^{n+1}\,\Phi_{\varphi,n-j}(K,L)^{-n}=
  \frac{\varphi'(1^-)}{j}\,\lim_{\varepsilon\to0^+}
  \frac{\Phi_{n-j}(K+_\varphi\varepsilon\cdot L)
  -\Phi_{n-j}(K)}{\varepsilon}.
 \end{align}
\end{proposition}
\begin{proof}
 By \eqref{eq.Orlicz-mv_1_*} we have
 \begin{align*}
  &\left.\frac{d}{d\varepsilon}\right|_{\varepsilon=0^+}
    \int_{G_{n,j}}\vol_j\big((K+_\varphi\varepsilon\cdot L)|\xi
    \big)^{-n}\,d\nu_{n,j}(\xi) \\
  &=-n\int_{G_{n,j}} \vol_j(K|\xi)^{-n-1}\,
    \left.\frac{d}{d\varepsilon}\right|_{\varepsilon=0^+}
    \vol_j\big((K+_\varphi\varepsilon\cdot L)
    |\xi\big)\,d\nu_{n,j}(\xi) \\
  &=-\frac{jn}{\varphi'(1^-)}\int_{G_{n,j}}
    V_\varphi^{(j)}\big(K|\xi,L|\xi\big)
    \,\vol_j \big(K|\xi\big)^{-n-1}\,d\nu_{n,j}(\xi) \\
  &=-\frac{jn}{\varphi'(1^-)}
    \left(\frac{\omega_n}{\omega_j}\right)^n\,\Phi_{\varphi,n-j}(K,L)^{-n}.
 \end{align*}
 Thus,
 \begin{align*}
  &\lim_{\varepsilon\to0^+}
  \frac{\Phi_{n-j}(K+_\varphi\varepsilon\cdot\sub{^\varphi} L)
  -\Phi_{n-j}(K)}{\varepsilon} \\
  &\quad=\left.\frac{d}{d\varepsilon}
   \Phi_{n-j}(K+_\varphi\varepsilon\cdot L)\right|_{\varepsilon=0^+} \\
  &\quad=\frac{\omega_n}{\omega_j}\,
    \left.\frac{d}{d\varepsilon}\right|_{\varepsilon=0^+} 
    \left(\int_{G_{n,j}}\vol_j\big((K+_\varphi\varepsilon\cdot L)|\xi
    \big)^{-n}\,d\nu_{n,j}(\xi)\right)^{-\frac1n} \\
  &\quad=-\frac{\omega_n}{\omega_j}\,\frac1n\left(\int_{G_{n,j}}
    \vol_j\big(K|\xi\big)^{-n}\,d\nu_{n,j}(\xi)\right)^{-\frac1n-1} \\
  &\qquad\quad \left.\frac{d}{d\varepsilon}\right|_{\varepsilon=0^+}
    \int_{G_{n,j}}\vol_j\big((K+_\varphi\varepsilon\cdot L)|\xi
    \big)^{-n}\,d\nu_{n,j}(\xi) \\
  &\quad=\frac{j}{\varphi'(1^-)}\left[
    \frac{\omega_n}{\omega_j}\,\left(
    \int_{G_{n,j}}\vol_j\big(K|\xi\big)^{-n}
    \,d\nu_{n,j}(\xi)\right)^{-\frac1n}\right]^{n+1}
    \,\Phi_{\varphi,n-j}(K,L)^{-n} \\  
  &\quad=\frac{j}{\varphi'(1^-)}\;\Phi_{n-j}(K)^{n+1}\;\Phi_{\varphi,n-j}(K,L)^{-n}.
 \end{align*}
\end{proof}

\medskip

\noindent {\bf Note}. Definition \ref{def.O-Aff-Querm} for $j=n$ gives that
\begin{align}\label{eq.OMAQ_0}
 \Phi_{\varphi,0}(K,L)^{-n}= V_{\varphi}(K,L)\,V(K)^{-n-1}.
\end{align}
So in that case, formula \eqref{eq.1st_variation} reads exactly 
as the corresponding first variation formula \eqref{eq.Orlicz-mv_1_*} for the 
Orlicz mixed volume.

\medskip

The following lemma comes from \cite[Theorem 5.2]{GHW1}. 

\begin{lemma}\label{lemma.OMAQ-linear}
 Let $\varphi\in\mathscr{C}$, $K,L\in{\cal K}_{o}^n$, $\varepsilon>0$ 
 and $T\in GL(n)$. Then,
 \begin{align*}
  T(K+_\varphi\varepsilon\cdot L) = TK +_\varphi \varepsilon\cdot TL.
 \end{align*}
\end{lemma}
\begin{proof}
 Note that for any $u\in\IS^{n-1}$ and $Q\in{\cal K}_o^n$,
 \begin{align*}
  h_{TQ}(u)=h_Q(T^*u).
 \end{align*}
 Thus, by definition \eqref{eq.Orlicz_sum} for the Orlicz linear combination, 
 we have that for every $u\in\IS^{n-1}$.
 \begin{align*}
  h_{TK +_\varphi \varepsilon\cdot TL}(u)
  &=\inf\left\{ \lambda>0:\,
    \varphi\left(\frac{h_{TK}(u)}{\lambda}\right)
    +\varepsilon\,\varphi\left(\frac{h_{TL}(u)}
    {\lambda}\right)\leq 1\right\} \\
  &=\inf\left\{ \lambda>0:\,
    \varphi\left(\frac{h_{K}(T^*u)}{\lambda}\right)
    +\varepsilon\,\varphi\left(\frac{h_{L}(T^*u)}
    {\lambda}\right)\leq 1\right\} \\
  &=h_{K +_\varphi \varepsilon\cdot L}(T^*u) 
   =h_{T(K +_\varphi \varepsilon\cdot L)}(u)
 \end{align*}
\end{proof}

Using the first variation formula \eqref{eq.1st_variation}, we can easily
see that Orlicz mixed affine quermassintegrals are invariant under volume 
preserving linear transportations.

\begin{proposition}\label{prop.OMAQ-affine}
 Let $\varphi\in\mathscr{C}$, $K,L\in{\cal K}_{o}^n$, $1\leq j \leq n$ 
 and $T\in SL(n)$. Then,
 \begin{align*}
  \Phi_{\varphi,n-j}(TK,TL)=\Phi_{\varphi,n-j}(K,L).
 \end{align*}
\end{proposition}
\begin{proof}
 By Proposition \ref{prop.1st_variation}, Lemma \ref{lemma.OMAQ-linear} 
 and the $SL(n)$-invariant of Lutwak's affine quermassintegrals, we get that
 \begin{align*}
  \Phi_{\varphi,n-j}(TK,TL)^{-n}
  &= \Phi_{n-j}(TK)^{-n-1}\,
   \frac{\varphi'(1^-)}{j}\,
   \left.\frac{d}{d\varepsilon}\right|_{\varepsilon=0^+}
   \Phi_{n-j}(TK+_\varphi\varepsilon\cdot\sub{^\varphi} TL) \\
  &=\Phi_{n-j}(K)^{-n-1}\,
   \frac{\varphi'(1^-)}{j}\,
   \left.\frac{d}{d\varepsilon}\right|_{\varepsilon=0^+}
   \Phi_{n-j}(K+_\varphi\varepsilon\cdot\sub{^\varphi} L) \\
  &=\Phi_{\varphi,n-j}(K,L)^{-n}.
 \end{align*}
\end{proof}

\medskip
 
We close the section with a definition, that comes of by choosing $\varphi(t)=t^p$
in \eqref{eq.O-Aff-Querm}.

\begin{definition}
 The \textit{$L_p$ mixed affine quermassintegrals} of $K,L\in{\cal K}_o^n$,
 $p\geq1$, are defined by 
 \begin{align}\label{eq.L_pMAQ}
  \Phi_{p,n-j}(K,L):=\frac{\omega_n}{\omega_j}\left(\int_{G_{n,j}}
  V_p^{(j)}\big(K|\xi,L|\xi\big)\,\vol_j
  \big(K|\xi\big)^{-n-1}\,d\nu_{n,j}(\xi)\right)^{-\frac1n}.
 \end{align}
 In particular, the \textit{mixed affine quermassintegrals} of 
 $K,L\in{\cal K}_o^n$ are defined by
 \begin{align}\label{eq.MAQ}
  \Phi_{1,n-j}(K,L):=\frac{\omega_n}{\omega_j}\left(\int_{G_{n,j}}
  V_1^{(j)}\big(K|\xi,L|\xi\big)\,\vol_j
  \big(K|\xi\big)^{-n-1}\,d\nu_{n,j}(\xi)\right)^{-\frac1n}.
 \end{align}
\end{definition}

\section{Orlicz-Minkowski Inequality \\ for
         Orlicz Mixed Affine Quermassintegrals.}\label{sec.O-M_OMAQ}

In this section we prove an Orlicz-Minkowski inequality for the Orlicz 
mixed affine quermassintegrals. For its proof we use the Orlicz-Minkowski 
inequality \eqref{eq.Orlicz-Mink_ineq} and H\"{o}lder inequality, which we 
quote here for the reader's convenience (see \cite[Theorem 189]{HLP}).

\medskip

\noindent \textbf{Theorem} (H\"{o}lder's inequality).
 Let $f,g:X\to[0,\infty]$ be measurable functions on 
a measure space $(X,\mu)$. For every $p\neq0$ we consider $p'\neq0$ such that 
$\frac1{p}+\frac1{p'}=1$.
\begin{itemize}
 \item[(i)] If $p\geq1$, then
            \begin{align}\label{eq.holder_i}
             \int fg\,d\mu \leq 
             \left(\int f^p\,d\mu\right)^{1/p} 
             \left(\int g^{p'}\,d\mu\right)^{1/p'},
            \end{align}
            with equality if and only if $f^p$ and $g^{p'}$ are proportional.
 \item[(ii)] If $0<p<1$ or $p<0$, then
            \begin{align}\label{eq.holder_ii}
             \int fg\,d\mu \geq 
             \left(\int f^p\,d\mu\right)^{1/p} 
             \left(\int g^{p'}\,d\mu\right)^{1/p'},
            \end{align}
            with equality if and only if $f^p$ and $g^{p'}$ are proportional,
            or $fg\equiv0$. 
\end{itemize}

\medskip
 
\begin{theorem}[Orlicz-Minkowski inequality 
for Orlicz mixed affine quermassintegrals] \label{thm.O-M_I_AQ} 
 Let $K,L\in{\cal K}_{o}^n$, $\varphi\in\mathscr{C}$, and $1\leq j \leq n$. 
 Then,
 \begin{align}\label{eq.O-M_I_AQ}
  \left(\frac{\Phi_{\varphi,n-j}(K,L)}{\Phi_{n-j}(K)}\right)^{-n}\geq 
  \varphi\left(\left(\frac{\Phi_{n-j}(L)}{\Phi_{n-j}(K)}\right)^{1/j}\right).
 \end{align}
 If $\varphi$ is strictly convex, then equality holds 
 if and only if $K$ and $L$ are dilates of each other.
\end{theorem}
\begin{proof}
 For every $K\in{\cal K}_{o}^n$, we define the Borer probability measure 
 $\mu_{K,n,j}$ on $G_{n,j}$ by
 \begin{align*}
  d\mu_{K,n,j}(\xi)
  =\frac{\vol_j(K|\xi)^{-n}}{\Phi_{n-j}(K)^{-n}}\,d\nu_{n,j}(\xi).
 \end{align*}
 Then, Orlicz-Minkowski inequality \eqref{eq.Orlicz-Mink_ineq}, and Jensen
 inequality for $\mu_{n,j}$, imply
 \begin{align}\label{eq.*}
  \left(\frac{\Phi_{\varphi,n-j}(K,L)}{\Phi_{n-j}(K)}\right)^{-n}
  &=\int_{G_{n,j}}\frac{V_\varphi^{(j)}(K|\xi,L|\xi)}
    {\vol_j(K|\xi)}\,d\mu_{K,n,j}(\xi) \nonumber \\
  &\geq \int_{G_{n,j}}\varphi\left(\left(\frac{\vol_j(L|\xi)}
    {\vol_j(K|\xi)}\right)^{1/j}\right) d\mu_{K,n,j}(\xi) \nonumber \\
  &\geq \varphi\left(\int_{G_{n,j}}\left(\frac{\vol_j(L|\xi)}
    {\vol_j(K|\xi)}\right)^{1/j} d\mu_{K,n,j}(\xi)\right) \nonumber \\
  &=\varphi\left(\frac{\int_{G_{n,j}}
    \vol_j(K|\xi)^{-\frac{jn+1}{j}}\,\vol_j(L|\xi)^\frac1j\,d\nu_{n,j}}
    {\Phi_{n-j}(K)^{-n}}\right).
 \end{align}
 We use H\"older inequality \eqref{eq.holder_ii} on  $G_{n,j}$, with 
 exponents 
 \begin{align*}
  p=\frac{jn}{jn+1} \qquad{\rm and}\qquad p'=-jn<0.
 \end{align*}
 Taking into account that $\varphi$ is increasing, we get that
 \begin{align*}
  &\int_{G_{n,j}} \vol_j(K|\xi)^{-\frac{jn+1}{j}}
   \,\vol_j(L|\xi)^\frac1j\,d\nu_{n,j} \nonumber \\ 
  &\qquad\qquad\geq 
    \left(\int_{G_{n,j}} \vol_j(K|\xi)^{-n}\,d\nu_{n,j}\right)^\frac{jn+1}{jn}\;
    \left(\int_{G_{n,j}} \vol_j(L|\xi)^{-n}\,d\nu_{n,j}\right)^{-\frac{1}{jn}} \\
  &\qquad\qquad= \Phi_{n-j}(K)^{-n-\frac1j}\;\Phi_{n-j}(L)^{\frac1j}.
 \end{align*}
 Thus, by \eqref{eq.*} we have
 \begin{align*}
  \left(\frac{\Phi_{\varphi,n-j}(K,L)}{\Phi_{n-j}(K)}\right)^{-n} 
  \geq\varphi\left(\frac{\Phi_{n-j}(K)^{-n-\frac1j}\;\Phi_{n-j}(L)^{\frac1j}}
    {\Phi_{n-j}(K)^{-n}}\right) 
  =\varphi\left(\left(\frac{\Phi_{n-j}(L)}{\Phi_{n-j}(K)}\right)^{1/j}\right).
 \end{align*}
 
 For the equality condition, note that if $K,L$ are dilates of each other, 
 then it can be easily checked that equality holds in \eqref{eq.O-M_I_AQ}. 
 Conversely, if we assume that $\varphi$ is strictly convex and that 
 equality holds in \eqref{eq.O-M_I_AQ}, then all inequalities in the 
 above proof should hold as equalities. Thus, equality must holds in the 
 Orlicz-Minkowski inequality for $K|\xi$ and $L|\xi$, $\xi\in G_{n,j}$, 
 and so we must have that for every $\xi\in G_{n,j}$ there exist 
 $\lambda(\xi)>0$ such that
 \begin{align}\label{eq.equal_1}
  L|\xi=\lambda(\xi)\,L|\xi \qquad \forall\,\xi\in G_{n,j}.
 \end{align}
 Moreover we must have equality in Jensen's and H\"{o}lder's inequalities.
 This implies that the positive functions $f(\xi):=\vol_j(K|\xi)$ and 
 $g(\xi):=\vol_j(L|\xi)$, $\xi\in G_{n,j}$, must be proportional to each 
 other, i.e., there exists $\lambda>0$ such that 
 \begin{align}\label{eq.equal_2}
  \vol_j(L|\xi)=\lambda\,\vol_j(K|\xi) \qquad \forall\,\xi\in G_{n,j}.
 \end{align}
 By \eqref{eq.equal_1} and \eqref{eq.equal_2} we conclude that $L=\lambda K$.
\end{proof}

Next uniqueness criterion follows directly from Theorem \ref{thm.O-M_I_AQ}.

\begin{proposition}\label{thm.crit}
 Let $\varphi\in\mathscr{C}$ be strictly convex, $1\leq j\leq n$, and 
 $K,L\in{\cal M}^n\subseteq{\cal K}^n_o$. If
 \begin{align}\label{eq.crit_i}
  \Phi_{\varphi,n-j}(M,K)=\Phi_{\varphi,n-j}(M,L),
  \quad \forall\,M\in{\cal M}^n
 \end{align}
 or
 \begin{align}\label{eq.crit_ii}
  \frac{\Phi_{\varphi,n-j}(K,M)}{\Phi_{n-j}(K)}
  =\frac{\Phi_{\varphi,n-j}(L,M)}{\Phi_{n-j}(L)},
  \quad \forall\,M\in{\cal M}^n,
 \end{align}
 then $K=L$.
\end{proposition}
\begin{proof}
 First we suppose that \eqref{eq.crit_i} holds, and we take $M=K$. Then by 
 \eqref{eq.O-M_I_AQ} we have
 \begin{align*}
  \Phi_{n-j}(K)^{-n}=\Phi_{\varphi,n-j}(K,L)^{-n}\geq\Phi_{n-j}(K)^{-n}\,
  \varphi\left(\left(\frac{\Phi_{n-j}(L)}{\Phi_{n-j}(K)}\right)^{1/j}\right),
 \end{align*}
 with equality if and only if $K$ and $L$ are dilates of each other. Thus we have,
 \begin{align*}
  \varphi\left(\left(\frac{\Phi_{n-j}(L)}{\Phi_{n-j}(K)}\right)^{1/j}\right)
  \leq 1,
 \end{align*}
 with equality if and only if $K$ and $L$ are dilates of each other. Since 
 $\varphi$ is increasing and $\varphi(1)=1$, we get that $\Phi_{n-j}(L) 
 \leq \Phi_{n-j}(K)$, with equality if and only if $K$ and $L$ are dilates 
 of each other. Under the light of $j$-homogeneity of the affine 
 quermassintegrals, this means that 
 \begin{align*}
  \Phi_{n-j}(L) \leq \Phi_{n-j}(K),
 \end{align*}
 with equality if and only if $K=L$. Similarly, by taking $M=L$ in 
 \eqref{eq.crit_i}, we get
 \begin{align*}
  \Phi_{n-j}(K) \leq \Phi_{n-j}(L),
 \end{align*}
 with equality if and only if $K=L$. Thus, we must have 
 $\Phi_{n-j}(K)=\Phi_{n-j}(L)$  and $K=L$. The same arguments also show 
 that if \eqref{eq.crit_ii} holds, then $K=L$.
\end{proof}

\section{Olricz-Brunn-Minkowski Inequality \\
         for Affine Quermassintegrals.}\label{sec.O-B-M_AQ}

\begin{lemma}
 Let $K,L\in{\cal K}_{o}^n$, $1\leq j \leq n$, and $\varphi\in\mathscr{C}$. 
 Then, for every $\varepsilon>0$ 
 \begin{align}\label{eq.1}
  1=\left(\frac{\Phi_{\varphi,n-j}(K\os\varepsilon\cdot L,K)}
  {\Phi_{n-j}(K\os\varepsilon\cdot L)}\right)^{-n} 
  +\varepsilon\left(\frac{\Phi_{\varphi,n-j}(K\os\varepsilon\cdot L,L)}
  {\Phi_{n-j}(K\os\varepsilon\cdot L)}\right)^{-n}.
 \end{align}
\end{lemma}
\begin{proof}
 We first prove the following fact:
 For every $A,B\in{\cal K}^j$, $1\leq j \leq n$ and 
 $\varepsilon>0$ one has that
 \begin{align*}
  V_\varphi(A\os\varepsilon\cdot B,A)
  +\varepsilon V_\varphi(A\os\varepsilon\cdot B,B)=V(A\os\varepsilon\cdot B).
 \end{align*}
 Indeed, if $A_\varphi:=A\os\varepsilon\cdot B$, then
 \eqref{eq.Orlicz-mv_1} and \eqref{eq.Orlicz_sum_*} imply that
 \begin{align*}
  &V_\varphi(A_\varphi,A) 
     + \varepsilon V_\varphi(A_\varphi,B) \\
  &\quad=\frac1n\int_{\IS^{n-1}}
     \varphi\left(\frac{h(A,u)}{h(A_\varphi,u)}\right)
    \,h(A_\varphi,u)\,dS(A_\varphi,u) \\
  &\quad\quad+\frac{\varepsilon}n\int_{\IS^{n-1}}
    \varphi\left(\frac{h(B,u)}{h(A_\varphi,u)}\right)
    \,h(A_\varphi,u)\,dS(A_\varphi,u) \\
  &\quad=\frac1n\int_{\IS^{n-1}}\left[
    \varphi\left(\frac{h(A,u)}{h(A_\varphi,u)}\right)
    +\varepsilon\varphi\left(\frac{h(B,u)}
     {h(A_\varphi,u)}\right)\right]
    \,h(A_\varphi,u)\,dS(A_\varphi,u) \\
  &\quad=\frac1n\int_{\IS^{n-1}}
    \,h(A_\varphi,u)\,dS(A_\varphi,u)=V(A_\varphi).
 \end{align*}
 Thus, by setting $K_\varphi=K\os\varepsilon\cdot L$, we get
 \begin{align*}
  &\Phi_{\varphi,n-j}(K_\varphi,K)^{-n}
   +\varepsilon \Phi_{\varphi,n-j}(K_\varphi,L)^{-n} \\
  &\quad=\left(\frac{\omega_n}{\omega_j}\right)^{-n}
   \int_{G_{n,j}}V_\varphi^{(j)}(K_\varphi|\xi,K|\xi)
   -\varepsilon V_\varphi^{(j)}(K_\varphi|\xi,L|\xi)
   \;\vol_j(K_\varphi|\xi)^{-n-1}\,d\nu_{n,j}(\xi) \\
  &\quad=\left(\frac{\omega_n}{\omega_j}\right)^{-n}
    \int_{G_{n,j}}\vol_j(K_\varphi|\xi)^{-n}\,d\nu_{n,j}(\xi) \\
  &\quad=\Phi_{n-j}(K_\varphi)^{-n}.
 \end{align*}
\end{proof}


\begin{theorem}[Orlicz-Brunn-Minkowski inequality for affine quermassintegrals]
\label{thm.O-B-M_I_AQ}$\hfill$ \\
 Let $K,L\in{\cal K}_{o}^n$, $1\leq j \leq n$, $\varphi\in\mathscr{C}$, and
 $\varepsilon>0$. Then,
 \begin{align}\label{eq.O-BM_ineq_Aff_Quer}
  1\geq 
  \varphi\left(\left(\frac{\Phi_{n-j}(K)}
  {\Phi_{n-j}(K\os\varepsilon\cdot L)}\right)^{1/j}\right)
  +\varepsilon
  \varphi\left(\left(\frac{\Phi_{n-j}(L)}
  {\Phi_{n-j}(K\os\varepsilon\cdot L)}\right)^{1/j}\right).
 \end{align}
 If in addition, $\varphi$ is strictly convex, then equality holds 
 if and only if $K$ and $L$ are dilates of each other. 
\end{theorem}
\begin{proof}
 By \eqref{eq.1} and the Orlicz-Minkowski inequality 
 \eqref{eq.O-M_I_AQ} we have that
 \begin{align*}
  1
  &=\left(\frac{\Phi_{\varphi,n-j}(K\os\varepsilon\cdot L,K)}
   {\Phi_{n-j}(K\os\varepsilon\cdot L)}\right)^{-n} 
   +\varepsilon\left(\frac{\Phi_{\varphi,n-j}(K\os\varepsilon\cdot L,L)}
   {\Phi_{n-j}(K\os\varepsilon\cdot L)}\right)^{-n} \\
  &\geq \varphi\left(\left(\frac{\Phi_{n-j}(K)}
   {\Phi_{n-j}(K\os\varepsilon\cdot L)}\right)^{1/j}\right)
   +\varepsilon
   \varphi\left(\left(\frac{\Phi_{n-j}(L)}
   {\Phi_{n-j}(K\os\varepsilon\cdot L)}\right)^{1/j}\right).
 \end{align*}
 Let now suppose that $\varphi$ is strictly convex. Then, by the 
 equality conditions in the Orlicz-Minkowski inequality \eqref{eq.O-M_I_AQ}, 
 and since $K,L\in{\cal K}_{o}^n$, we get that equality holds in 
 \eqref{eq.O-BM_ineq_Aff_Quer} if and only if $K$ and $K\os\varepsilon\cdot L$ 
 are dilates of each other and $L$ and $K\os\varepsilon\cdot L$ are dilates 
 of each other. Thus, equality holds if and only if $K$ and $L$ are 
 dilates of each other.
\end{proof}

\medskip

In the last proof, we saw that Theorem \ref{thm.O-B-M_I_AQ} is a 
consequence of Theorem \ref{thm.O-M_I_AQ}. Actually, those two inequalities 
are equivalent.

\medskip

\begin{proposition} 
 Inequalities \eqref{eq.O-BM_ineq_Aff_Quer} and \eqref{eq.O-M_I_AQ} are 
 equivalent. 
\end{proposition}
\begin{proof}
 We only have to show that \eqref{eq.O-BM_ineq_Aff_Quer} implies 
 \eqref{eq.O-M_I_AQ}. Indeed,  by Proposition \ref{prop.1st_variation},
 Orlicz-Brunn-Minkowski inequality \eqref{eq.O-BM_ineq_Aff_Quer}, and 
 lemma \eqref{eq.O_lim}, we have
 \begin{align*}
  &\frac{j}{\varphi'(1^-)}\,\Phi_{n-j}(K)^{n+1}\,\Phi_{\varphi,n-j}(K,L)^{-n} 
  =\lim_{\varepsilon\to0^+}
    \frac{\Phi_{n-j}(K+_\varphi\varepsilon\cdot L)
    -\Phi_{n-j}(K)}{\varepsilon} \\
  &=\lim_{\varepsilon\to0^+}
    \frac{1-\varphi\left(\left(\frac{\Phi_{n-j}(K)}
    {\Phi_{n-j}(K+_\varphi\varepsilon\cdot L)}\right)^\frac1j\right)}
    {\varepsilon} 
    \;\cdot\;
    \lim_{\varepsilon\to0^+}\frac{1-\frac{\Phi_{n-j}(K)}
    {\Phi_{n-j}(K+_\varphi\varepsilon\cdot L)}}
    {1-\varphi\left(\left(\frac{\Phi_{n-j}(K)}
    {\Phi_{n-j}(K+_\varphi\varepsilon\cdot L)}\right)^\frac1j\right)} \\
    &\;\;\,\cdot 
    \lim_{\varepsilon\to0^+}\Phi_{n-j}(K+_\varphi\varepsilon\cdot L) \\
  &\geq \lim_{\varepsilon\to0^+}
    \varphi\left(\left(\frac{\Phi_{n-j}(L)}
    {\Phi_{n-j}(K+_\varphi\varepsilon\cdot L)}\right)^\frac1j\right)
    \;\cdot\;
    \lim_{t\to1^-}\frac{1-t}{1-\varphi\big(t^{1/j}\big)}  
    \;\cdot\; \lim_{\varepsilon\to0^+}\Phi_{n-j}(K+_\varphi\varepsilon\cdot L) \\
  &= \varphi\left(\left(\frac{\Phi_{n-j}(L)}
    {\Phi_{n-j}(K)}\right)^\frac1j\right)
    \;\cdot\;\frac{j}{\varphi'(1^-)} \;\cdot\; \Phi_{n-j}(K).
 \end{align*}
\end{proof}



\vspace{2em}

\noindent Nikos Dafnis

\smallskip

\noindent 
Vienna University of Technology, \\
Institute of Discrete Mathematics and Geometry. \\
Wiedner Hauptstrasse 8-10, 1040 Vienna, Austria. \\
Email address: nikdafnis@gmail.com


\footnotesize
\begin{thebibliography}
\footnotesize

\bibitem{Aleks}{A. D. Aleksandrov}. \textit{On the theory of mixed volumes. 
I. Extension of certain concepts in the theory of convex bodies}. Mat. Sb. 
N.S. {\bf 2} (1937), 947-972. [in Russian]

\bibitem{DP}{N. Dafnis \& G. Paouris}. \textit{Estimates for the affine
and dual affine quermassintegrals of convex bodies}. Illinois Journal of
Mathematics {\bf 56}, no 4 (2012) 1005-1021.

\bibitem{FJ}{W. Fenchel and B. Jessen}. \textit{Mengenfunktionen und konvexe 
K\"orper}. Danske Vid. Selskab. Mat.-fys. Medd. {\bf 16} (1938), 1-31.

\bibitem{Firey1}{W. J. Firey}. \textit{Polar means of convex bodies and a dual 
to the Brunn-Minkowski theorem}. Canadian Journal of Mathematics {\bf 13} 
(1961), 444-453.

\bibitem{Firey2}{W. J. Firey}. \textit{$p$-means of convex bodies}.
Math. Scand. {\bf 10} (1962), 17-24.

\bibitem{Gardner}{R. J. Gardner}. \textit{Geometric Tomography},
Encyclopedia of Mathematics and its Applications {\bf 58},
Cambridge University Press, Cambridge 2nd edition (2006).

\bibitem{GHW1}{R. J. Gardner, D. Hug \& W. Weil}. 
\textit{The Orlicz-Brunn-Minkowski theory: A general framework, additions, 
and inequalities}. Journal of Differential Geometry {\bf 97} (2014) 427-476.

\bibitem{GHW2}{R. J. Gardner, D. Hug \& W. Weil}. 
\textit{Dual Orlicz-Brunn-Minkowski theory}. Advances in Mathematics 
{\bf 264} (2014), 700-725.

\bibitem{Grinberg}{E. L. Grinberg}.
\textit{Isoperimetric inequalities and identities for $k$-dimensional 
cross-sections of a convex bodies}. Mathematische Annalen {\bf 291}
(1991) no. 1, 75-86.

\bibitem{HLP}{G. Hardy, J. Littlewood \& G. Pόlya}. \textit{Inequalities}. 
Cambridge University Press, Cambridge, 1934.

\bibitem{LZX}{D.Y. Li, D. Zou \& G.Xiong}. \textit{Orlicz mixed affine 
quermassintegrals}. Sci China Math, {\bf 58} (2015) 17151722, doi 
10.1007/s11425-014-4965-1.

\bibitem {Lutwak1}{E. Lutwak}. \textit{A general isepiphanic inequality},
Proceedings of the American Mathematical Society {\bf 90} (1984), 415-421.

\bibitem{Lutwak2}{E. Lutwak}. \textit{Intersection bodies and dual mixed
volumes}, Advances in Mathematics {\bf 71} (1988), 232-261.

\bibitem{Lutwak3}{E. Lutwak}. \textit{Dual mixed volumes}. Pacific Journal
of Mathematics {\bf 58} (1975) 531-538.

\bibitem{Lutwak4}{E. Lutwak}. \textit{Intersection bodies and dual mixed
volumes}. Advances in Mathematics {\bf 71} (1988) 232-261.

\bibitem{Lutwak7}{E. Lutwak}. \textit{The Brunn-Minkowski-Firey Theory I: 
Mixed volumes and the Minkowski Problem}. J. Differential Geom. {\bf 38} 
(1993), 131-150.

\bibitem{Lutwak5}{E. Lutwak}. \textit{The Brunn-Minkowski-Firey theory II:
affine and geominimal surface areas}. Advances in Mathematics {\bf 118} 
(1996) 244-294.

\bibitem{Lutwak6}{E. Lutwak}. \textit{Extended affine surface area}.
Advances in Mathematics {\bf 85} (1991), 39-68.

\bibitem{Lutwak8}{E. Lutwak}. \textit{ Inequalities for Hadwiger's
harmonic Quermassintegrals}. Mathematische Annalen {\bf 280} (1988),
165-175.

\bibitem{LYZ6}{E. Lutwak, D. Yang \& G. Zhang}.
\textit{Orlicz projection bodies}. Advances in Mathematics {\bf 223} (2010), 
220-242.

\bibitem{LYZ2}{E. Lutwak, D. Yang \& G. Zhang}.
\textit{Orlicz centroid bodies}. Journal of Differential Geometry
{\bf 84} (2010), 365-387.

\bibitem{LYZ1}{E. Lutwak, D. Yang, G. Zhang}.
\textit{The Brunn-Minkowski-Firey inequality for non-convex sets}. Advances 
in Applied Mathematics {\bf 48} (2012), 407-413.

\bibitem{PP}{G. Paouris \& P. Pivovarov}. \textit{Small ball probabilities 
for the volume of random convex sets}. Discrete and Comp. Geom. {\bf 49} 
no. 3 (2013), 601-646.

\bibitem{Petty1}{C. M. Petty}. \textit{Affine isoperimetric problems}. 
Ann. N.Y. Acad. Sci. {\bf 440} (1985), 113-127.

\bibitem{Petty2}{C. M. Petty}. \textit{Geominimal surface area}. 
Geom. Dedicata {\bf 3} (1974), 77-97.

\bibitem{Schneider}{R. Schneider}. \textit{Convex bodies: the Brunn-Minkowski
theory}, Second Expanded Edition  Encyclopedia of Mathematics and its
applications, Cambridge University Press (2014).

\bibitem{Schutt}{C. Sch\"utt} \textit{On the affine surface area}.
Proc. Amer. Math. Soc. {\bf 118} (1993), 1213-1218.

\bibitem{SW}{C. Sch\"utt and E. Werner} \textit{The convex floating body}. 
Math. Scand. {\bf 66} (1990), 275-290.

\bibitem{Werner}{E. Werner}. \textit{Illumination bodies and affine 
surface area}. Studia Math. {\bf 110} (1994), 257-269.

\bibitem{Zhao}{Chang-Jian Zhao}. \textit{Orlicz dual affine quermassintegrals}.
Forum Mathematicum 2017, \\ 
DOI: https://doi.org/10.1515/forum-2017-0174.

\bibitem{ZZX}{B. Zhu, J. Zhou, W. Xu}. \textit{Dual Orlicz-Brunn-Minkowski theory}.
Advances in Mathematics {\bf 264}, (2014) 700-725.
 
\end{thebibliography}
\end{document}